\documentclass[a4paper,pdftex,reqno,10pt]{amsart}

\allowdisplaybreaks[1]

\usepackage{geometry}

\usepackage{graphicx}
\usepackage{enumerate}
\usepackage{courier}
\usepackage{times}
\usepackage{amsmath,amssymb}

\usepackage{hyperref}

\theoremstyle{plain}
\newtheorem{Theorem}{Theorem}[section]

\newtheorem{Corollary}[Theorem]{Corollary}
\newtheorem{Lemma}[Theorem]{Lemma}

\newenvironment{Proof}
{\begin{trivlist}\item[]{{\sc Proof.}}}{\hfill{$\square$}\noindent\end{trivlist}}

\theoremstyle{definition}
\newtheorem{Definition}[Theorem]{Definition}

\theoremstyle{remark}

\begin{document}


\title[The lengths of projective triply-even binary codes]{The lengths of projective triply-even binary codes}


 \author{Thomas Honold, Michael Kiermaier, Sascha Kurz, and Alfred Wassermann}
 \address{Thomas Honold, ZJU-UIUC Institute, Zhejiang University, 314400 Haining, China.}
 \email{honold@zju.edu.cn} 
 \address{Michael Kiermaier, University of Bayreuth, 95440 Bayreuth, Germany}
 \email{michael.kiermaier@uni-bayreuth.de}
 \address{Sascha Kurz, University of Bayreuth, 95440 Bayreuth, Germany}
 \email{sascha.kurz@uni-bayreuth.de}
 \address{Alfred Wassermann, University of Bayreuth, 95440 Bayreuth, Germany}
 \email{alfred.wassermann@uni-bayreuth.de}

\abstract{
It is shown that there does not exist a binary projective triply-even code of length $59$. This settles the last open length for projective triply-even binary codes. 
Therefore, projective triply-even binary codes exist precisely for lengths $15$, $16$, $30$, $31$, $32$, $45$--$51$, and $\ge 60$.\\[2mm]
}}

\maketitle

\section{Introduction}
Doubly-even codes were subject to extensive research in the last years. For applications and enumeration results 
we refer e.g.\ to \cite{doran2011codes}. More recently, triply-even codes where studied, see e.g.\ \cite{betsumiya2012triply}.  
These two classes of binary linear codes are special cases of so-called $\Delta$-divisible codes, where all weights of a $q$-ary linear code 
$\mathcal{C}$ are divisible by $\Delta$, see e.g.\ \cite{ward1999introduction}. 
The columns of a $k\times n$ generator matrix of $\mathcal{C}$ generate $n$ one-dimensional subspaces of $\mathbb{F}_q^k$ that are also called \emph{points} in Projective 
Geometry, see e.g.~\cite{dodunekov1998codes} or \cite[Chapter 17]{bierbrauer2004introduction}. The $\Delta$-divisibility of the linear code $\mathcal{C}$ translates as follows 
to the multiset $\mathcal{P}$ of points in 
$\mathbb{F}_q^k$. For each hyperplane $H$ of $\mathbb{F}_q^k$ we have $\# (\mathcal{P}\cap H)\equiv \#\mathcal{P} \pmod \Delta$. The set 
$\mathcal{P}$ is also called $\Delta$-divisible then. We remark that the hyperplanes correspond to the codewords of $\mathcal{C}$. A recent application of $\Delta$-divisible 
codes and sets is the maximum possible cardinality of partial $k$-spreads in $\mathbb{F}_q^v$, i.e., sets of $k$-dimensional subspaces in 
$\mathbb{F}_q^v$ with pairwise trivial intersection, see e.g.\ \cite{honold2018partial,kurz2017packing}. Due to the intersection property, 
every point of $\mathbb{F}_q^v$ is covered by at most one element of a given partial $k$-spread. Calling every non-covered point a hole, we 
can state that the set of holes of a partial $k$-spread is $q^{k-1}$-divisible, see e.g.~\cite[Theorem 8]{honold2018partial} containing also a generalization to 
so-called vector space partitions.\footnote{In a special case, the divisibility of the set of holes was already used in \cite{beutelspacher1975partial} to determine 
an upper bound for the maximum cardinality of a partial $k$-spread.} So, from the non-existence of $q^{k-1}$-divisible sets (or projective $q^{k-1}$ divisible linear codes, 
since we have a set of holes in this application) of a suitable effective length $n$ one can conclude the non-existence of partial $k$-spreads in $\mathbb{F}_q^v$ of a certain cardinality. 
Indeed, all currently known upper bounds for partial $k$-spreads can be obtained from such non-existence results for divisible codes, see 
e.g.~\cite{honold2018partial,kurz2017packing}.

From an application point of view, $q^r$-divisible linear codes, where $r$ is some positive rational number, are of special interest. If 
$G_1$ is a generator matrix of a $\Delta$-divisible $[n_1,k_1]_q$ code and $G_2$ is the generator matrix of another $\Delta$-divisible 
$[n_2,k_2]_q$ code, then $\begin{pmatrix}G_1&0\\0&G_2\end{pmatrix}$ is the generator matrix of a $\Delta$-divisible $[n_1+n_2,k_1+k_2]_q$ 
code. Since the set of all points of a $k$-dimensional subspace of $\mathbb{F}_q^v$ is $q^{k-1}$-divisible, for each prime power $q$ and each 
$r\in\mathbb{Q}_{>0}$, such that $q^r\in\mathbb{N}$, there exists a finite set $\mathcal{F}_q(r)$ of integers that cannot be the cardinality of a $q^r$-divisible (multi)-set 
or effective length of a (projective) $q^r$-divisible linear code. For multisets of points, i.e., linear codes, the question is completely 
resolved in \cite[Theorem 4]{kiermaier2017improvement} for all integers $r$ and all prime powers $q$. For sets of points or projective 
$q^r$-divisible linear codes the question is more complicated. A partial answer has been given in \cite[Theorem 13]{honold2018partial}:
\begin{Theorem}$\,$\\[-4mm]
\label{thm_characterization}
\begin{enumerate}[(i)]
  \item\label{picture_q_2_r_1} $2^1$-divisible sets over $\mathbb{F}_2$ of
    cardinality $n$ exist for all $n\geq 3$ and do not exist for
    $n\in\{1,2\}$.
  \item\label{picture_q_2_r_2} $2^2$-divisible sets over $\mathbb{F}_2$ of
    cardinality $n$ exist for $n\in\{7,8\}$ and all $n\geq 14$, and do
    not exist in all other cases.
  \item\label{picture_q_2_r_3} $2^3$-divisible sets over $\mathbb{F}_2$ of
    cardinality $n$ exist for
    $n\in\{15,16,30,31,32,45,46,$ $47,48,49,50,51\}$, for all $n\geq 60$,
    and possibly for $n=59$; in all other cases they do not exist.
  \end{enumerate}
\end{Theorem}      

In part~(iii) the existence question for a binary projective $2^3$-divisible linear code remains undecided. The aim of this paper is 
to complete this characterization result. Indeed, we will show that $n=59$ is impossible. We remark that the distinction between the existence 
of projective and possibly non-projective $q^r$-divisible linear codes of a certain length plays indeed a role for e.g.\ upper bounds on the 
maximum possible cardinality of partial $k$-spreads. As an example, in \cite[Theorem 13]{honold2018partial}, see also \cite{kurz2017packing}, 
it is shown that no projective $2^3$-divisible linear code of length $52$ exists, while there are non-projective examples with these parameters. 
From this non-existence result for projective $q^r$-divisible codes we can conclude that there can be at most $132$ solids in $\mathbb{F}_2^{11}$ 
with pairwise trivial intersection, which is the tightest currently known upper bound. With a lower bound of $129$, this is the smallest open 
case for the maximum cardinality of partial $k$-spreads over $\mathbb{F}_2$.  

The remaining part of the paper is structured as follows. In Section~\ref{sec_preliminaries} we state the necessary preliminaries from 
coding theory before we prove the non-existence of a binary projective $2^3$-divisible linear code with effective length $n=59$ in 
Section~\ref{sec_no_59}.  
We close with a brief conclusion and some open problems in Section~\ref{sec_conclusion}.  

\section{Preliminaries}
\label{sec_preliminaries}

A projective linear code $\mathcal{C}$ over $\mathbb{F}_q$ is called $q^r$-divisible for some $r\in\mathbb{Q}_{>0}$, such that 
$q^r\in\mathbb{N}$\footnote{More precisely, this conditions says that $q^r$ should be a power of the field characteristic $p$. In \cite[Theorem 1]{ward1981divisible} 
it has been shown that $\Delta$-divisible codes where $\Delta$ is relatively prime to $q$ correspond to repetitions of smaller codes. Thus, it 
suffices to consider $\Delta=p^l$ for integers $l$.}, if the 
weight of each codeword is divisible by $q^r$. By $A_i[\mathcal{C}]$ we denote the number of codewords of weight exactly $i$. 
Note that $A_0[\mathcal{C}]=1$. Whenever the code is clear from the context, we write $A_i$ instead of $A_i[\mathcal{C}]$. By 
$n=n(\mathcal{C})$ we denote the length and by $k=k(\mathcal{C})$ the dimension of the linear code. Given our assumption that 
$\mathcal{C}$ is projective the length equals the effective length\footnote{Whenever we speak of the length of a code in this paper, 
we mean its effective length. So, $[n,k]_q$ codes are $k$-dimensional linear codes in $\mathbb{F}_q^n$ with effective length $n$.}, i.e., 
there are no zero-columns in the generator matrix of 
$\mathcal{C}$ and we have a one-to-one correspondence to a set of $n$ spanning points in $\mathbb{F}_q^k$. Without creating confusion 
we use the same symbol $\mathcal{C}$ also for sets of points. By $a_i$ we denote the number of hyperplanes of $\mathbb{F}_q^k$ containing exactly 
$i$ points. Here, we have the relation $a_i=A_{n-i}$ for $0\le i<n$. Denoting the dual of a code by $\mathcal{C}^\perp$ we write 
$B_i=B_i(\mathcal{C})$ for the number of codewords of weight $i$ of $\mathcal{C}^\perp$. Due to our assumption that $\mathcal{C}$ is a 
projective code we have $B_0=1$ and $B_1=B_2=0$. The well-known MacWilliams identities, see e.g.~\cite{macwilliams1977theory}, relate the 
$A_i$ with the $B_i$ as follows:
$$
  \sum_{j=0}^n K_i(j)A_j(\mathcal{C})=2^kB_i(\mathcal{C})\quad\text{for } 0\le i\le n,
$$   
where 
$$
  K_i(j)=\sum_{s=0}^n (-1)^s {{n-j}\choose{i-s}}{{j}\choose{s}} \quad\text{for } 0\le i\le n.
$$
Obviously, we have $\sum_{i=0}^n A_i=2^k$, which is indeed equivalent to the first ($i=0$) MacWilliams equation. 
The polynomial $w(\mathcal{C})=\sum\limits_{i=0}^n A_i(\mathcal{C}) x^i$ is called the weight enumerator of $\mathcal{C}$.

For a given $[n,k]_q$ code $\mathcal{C}$ and a codeword $c\in\mathcal{C}$ of weight $w$ we can consider the so-called residual code 
$\mathcal{C}_w$, which arises from $\mathcal{C}$ by restricting all codewords to those coordinates where $c$ has a zero entry. Thus, 
$\mathcal{C}_w$ is an $[n-w,\le k-1]_q$ code. If $\mathcal{C}$ is projective, then obviously also $\mathcal{C}_w$ is projective. Moreover, 
if $\mathcal{C}$ is $q^r$-divisible, then $\mathcal{C}_w$ is $q^{r-1}$-divisible, see e.g.\ \cite[Lemma 7]{honold2018partial}. Since 
the possible lengths of binary projective $2^2$-divisible linear codes are characterized in Theorem~\ref{thm_characterization}.(ii), we can state 
that the weights of a binary projective $2^3$-divisible linear code $\mathcal{C}_{59}$ of length $n=59$ have to be contained in $\{8,16,24,32,40\}$.  
For codewords of weight $40$ it is possible to characterize the possible corresponding residual codes: 
\begin{Lemma}
  \label{lemma_we_2_2_n_19}
  Let $\mathcal{C}$ be a projective $2^2$-divisible code of length $n=19$, then its weight enumerator is one of the following:
  \begin{itemize}
    \item $w(\mathcal{C})=1+1\cdot x^4+75\cdot x^8+51\cdot x^{12}$, i.e., $k=7$ and $B_3=5$;
    \item $w(\mathcal{C})=1+78\cdot x^8+48\cdot x^{12}+1\cdot x^{16}$, i.e., $k=7$ and $B_3=1$;
    \item $w(\mathcal{C})=1+4\cdot x^4+150\cdot x^8+100\cdot x^{12}+1\cdot x^{16}$, i.e., $k=8$ and $B_3=1$.
  \end{itemize}
\end{Lemma}
\begin{Proof}
  The first four MacWilliams identities directly yield the three mentioned cases and additionally the 
  weight enumerator $w(\mathcal{C})=1+5\cdot x^4+147\cdot x^8+103\cdot x^{12}$, i.e., $k=8$ (and $B_3=3$). Using the 
  MacWilliams transform we compute $B_4=2$. This contradicts $A_4\le B_4$, which is implied by self-orthogonality of $2^2$-divisible codes. 
\end{Proof}

Moreover, each of the three mentioned weight enumerators is attained by a projective $2^2$-divisible code of length $n=19$ having 
a nice geometric description:

Assume that $\mathcal{P}$ is a $4$-divisible set of points of size $n$ in some
ambient $\mathbb{F}_2$-vector space $V$.
For each plane $E$ intersecting $\mathcal{P}$ in a line $L$, $\mathcal{P}' =
(\mathcal{P} \setminus L) \cup (E\setminus L)$ is a $4$-divisible set of points of
size $n+1$.
In other words, the line $L$ is \emph{switched} to the affine plane $E\setminus L$.
The $4$-divisibility is seen by considering the characteristic functions of the
involved point sets, see \cite[Lemma~14]{honold2018partial}: We have
$\chi_{\mathcal{P}'} = \chi_{\mathcal{P}} + \chi_E - 2\chi_L$, where $\mathcal{P}$
and $E$ are $4$-divisible and $L$ is $2$-divisible.

In fact, all three types of projective $4$-divisible codes of length $19$ can be
constructed by switching four pairwise skew lines $L_1,\ldots,L_4$ in a solid $S$.
The choice of $L_1,\ldots,L_4$ is unique up to isomorphism and arises as the unique
line spread of $S$ without one of its lines.
In particular, each code is the union of the line $S\setminus\{L_1,L_2,L_3,L_4\}$
and $4$ affine planes.
For the explicit constructions below, we pick $L_1, L_2, L_3$ and $L_4$ as the row
spaces of the following matrices:
\begin{align*}
        L_1 & : \begin{pmatrix}
                1 & 0 & 0 & 0 \\
                0 & 1 & 0 & 0 
        \end{pmatrix} &
        L_2 & : \begin{pmatrix}
                1 & 0 & 1 & 0 \\
                0 & 1 & 0 & 1
        \end{pmatrix} &
        L_3 & : \begin{pmatrix}
                1 & 0 & 1 & 1 \\
                0 & 1 & 1 & 0
        \end{pmatrix} &
        L_4 & : \begin{pmatrix}
                1 & 0 & 0 & 1 \\
                0 & 1 & 1 & 1
        \end{pmatrix}
\end{align*}
The remaining three points of $S$ are the points on the line
\[
        \begin{pmatrix}
                0 & 0 & 1 & 0 \\
                0 & 0 & 0 & 1
        \end{pmatrix}\text{.}
\]
The lines will be embedded in the ambient space ($\mathbb{F}_2^7$ or $\mathbb{F}_2^8$) by adding all-zero coordinates.
Furthermore, the $i$th unit vector will be denoted by $e_i$ and for a vector $v$, the point generated by $v$ will be denoted by $\langle v\rangle$.
        
By the uniqueness of the choice of the $4$ lines, the result $\mathcal{P}$ of the
switching only depends on its image in $V/S$ of size $4$.
There are $3$ possible constellations of $4$ points:

\begin{itemize}
        \item Four points in general position, spanning a solid.
        The resulting code $C_1$ is of dimension $8$ and has weight distribution $(0^1 4^4
8^{150} 12^{100} 16^1)$ and $18432$ automorphisms.

        For an explicit construction, we switch in the planes $L_1 + \langle e_5\rangle$, $L_2 + \langle e_6\rangle$,
$L_3 + \langle e_7\rangle$ and $L_4 + \langle e_8\rangle$.
	This leads to the generator matrix
        \[
                \left(
                \begin{array}{ccccccccccccccccccc}
                0 & 0 & 0 & 0 & 1 & 0 & 1 & 0 & 1 & 0 & 1 & 0 & 1 & 0 & 1 & 0 & 1 & 0 & 1 \\
                0 & 0 & 0 & 0 & 0 & 1 & 1 & 0 & 0 & 1 & 1 & 0 & 0 & 1 & 1 & 0 & 0 & 1 & 1 \\
                1 & 0 & 1 & 0 & 0 & 0 & 0 & 0 & 1 & 0 & 1 & 0 & 1 & 1 & 0 & 0 & 0 & 1 & 1 \\
                0 & 1 & 1 & 0 & 0 & 0 & 0 & 0 & 0 & 1 & 1 & 0 & 1 & 0 & 1 & 0 & 1 & 1 & 0 \\
                0 & 0 & 0 & 1 & 1 & 1 & 1 & 0 & 0 & 0 & 0 & 0 & 0 & 0 & 0 & 0 & 0 & 0 & 0 \\
                0 & 0 & 0 & 0 & 0 & 0 & 0 & 1 & 1 & 1 & 1 & 0 & 0 & 0 & 0 & 0 & 0 & 0 & 0 \\
                0 & 0 & 0 & 0 & 0 & 0 & 0 & 0 & 0 & 0 & 0 & 1 & 1 & 1 & 1 & 0 & 0 & 0 & 0 \\
                0 & 0 & 0 & 0 & 0 & 0 & 0 & 0 & 0 & 0 & 0 & 0 & 0 & 0 & 0 & 1 & 1 & 1 & 1    
                \end{array}
                \right)\text{.}
        \]
        \item The three points on a line together with an additional point (i.e., the
complement of a triangle in a plane).
        The resulting code $C_2$ is of dimension $7$ and has weight distribution $(0^1 4^1
8^{75} 12^{51})$ and $1440$ automorphisms.

        Switching in the planes $L_1 + \langle e_5\rangle$, $L_2 + \langle e_6\rangle$, $L_3 + \langle e_7\rangle$, $L_4 + \langle e_6 + e_7\rangle$,
we get the explicit generator matrix
        \[
                \left(
                \begin{array}{ccccccccccccccccccc}
                0 & 0 & 0 & 0 & 1 & 0 & 1 & 0 & 1 & 0 & 1 & 0 & 1 & 0 & 1 & 0 & 1 & 0 & 1 \\
                0 & 0 & 0 & 0 & 0 & 1 & 1 & 0 & 0 & 1 & 1 & 0 & 0 & 1 & 1 & 0 & 0 & 1 & 1 \\
                1 & 0 & 1 & 0 & 0 & 0 & 0 & 0 & 1 & 0 & 1 & 0 & 1 & 1 & 0 & 0 & 0 & 1 & 1 \\
                0 & 1 & 1 & 0 & 0 & 0 & 0 & 0 & 0 & 1 & 1 & 0 & 1 & 0 & 1 & 0 & 1 & 1 & 0 \\
                0 & 0 & 0 & 1 & 1 & 1 & 1 & 0 & 0 & 0 & 0 & 0 & 0 & 0 & 0 & 0 & 0 & 0 & 0 \\
                0 & 0 & 0 & 0 & 0 & 0 & 0 & 1 & 1 & 1 & 1 & 0 & 0 & 0 & 0 & 1 & 1 & 1 & 1 \\
                0 & 0 & 0 & 0 & 0 & 0 & 0 & 0 & 0 & 0 & 0 & 1 & 1 & 1 & 1 & 1 & 1 & 1 & 1
                \end{array}
                \right)\text{.}
        \]

        The following alternative construction is worth noting:
        Switching three pairwise disjoint lines in $S$ such that the image modulo $S$ is a
line results in a binary $4$-divisible $[18,6]$-code.
        This code is the same as the concatenation of the $\mathbb{F}_4$-linear hexacode
(arising from a hyperoval over $\mathbb{F}_4$) with the binary $[3,2]$ simplex
code.
        Geometrically, it is the disjoint union of $6$ lines in an ambient space $H$ of
algebraic dimension $6$.
        Switching any of its lines always increases the ambient space dimension by one and
leads to the code $C_2$.
        The hyperplane $H$ corresponds $18 - 3 = 15$ points and belongs to the single
codeword of weight $4$.

        \item An affine plane (i.e., the complement of a line in a plane).
        The resulting code $C_3$ is of dimension $7$ and has weight distribution $(0^1
8^{78} 12^{48} 16^1)$ and $5760$ automorphisms.

        Switching in the planes $L_1 + \langle e_5\rangle$, $L_2 + \langle e_6\rangle$, $L_3 + \langle e_7\rangle$, $L_4 + \langle e_5 + e_6 + e_7\rangle$, we get the explicit generator matrix
        \[
                \left(
                \begin{array}{ccccccccccccccccccc}
                0 & 0 & 0 & 0 & 1 & 0 & 1 & 0 & 1 & 0 & 1 & 0 & 1 & 0 & 1 & 0 & 1 & 0 & 1 \\
                0 & 0 & 0 & 0 & 0 & 1 & 1 & 0 & 0 & 1 & 1 & 0 & 0 & 1 & 1 & 0 & 0 & 1 & 1 \\
                1 & 0 & 1 & 0 & 0 & 0 & 0 & 0 & 1 & 0 & 1 & 0 & 1 & 1 & 0 & 0 & 0 & 1 & 1 \\
                0 & 1 & 1 & 0 & 0 & 0 & 0 & 0 & 0 & 1 & 1 & 0 & 1 & 0 & 1 & 0 & 1 & 1 & 0 \\
                0 & 0 & 0 & 1 & 1 & 1 & 1 & 0 & 0 & 0 & 0 & 0 & 0 & 0 & 0 & 1 & 1 & 1 & 1 \\
                0 & 0 & 0 & 0 & 0 & 0 & 0 & 1 & 1 & 1 & 1 & 0 & 0 & 0 & 0 & 1 & 1 & 1 & 1 \\
                0 & 0 & 0 & 0 & 0 & 0 & 0 & 0 & 0 & 0 & 0 & 1 & 1 & 1 & 1 & 1 & 1 & 1 & 1
                \end{array}
                \right)\text{.}
        \]

        The code $C_3$ also arises from shortening the $[24,12]$ extended binary Golay code
$5$ times.
        Note that the result does not depend on the choice of the $5$ positions, as the
automorphism group of the Golay code (the Mathieu group $M_{24}$) acts $5$-fold
transitive on set of the coordinates.
\end{itemize}

\section{The non-existence of binary projective $2^3$-divisible linear codes of length $n=59$}
\label{sec_no_59}

Setting $y=2^k$, the first four MacWilliams identities yield
\begin{eqnarray*}
A_{16}[\mathcal{C}] &=& -10-4A_8-\frac{45}{4096}y+\frac{1}{4096}yB_3\\
A_{24}[\mathcal{C}] &=& 20+6A_8+\frac{1447}{4096}y-\frac{3}{4096}yB_3\\
A_{32}[\mathcal{C}] &=& -15-4A_8+\frac{2617}{4096}y+\frac{3}{4096}yB_3\\
A_{40}[\mathcal{C}] &=&  4+A_8+\frac{77}{4096}y-\frac{1}{4096}yB_3.
\end{eqnarray*}
Thus, we have $A_{16}+A_{40}=-6-3A_8+\frac{y}{128}\ge 0$, so that $y\ge 768$. (Commonly, we will write just $A_i$ instead of $A_i[\mathcal{C}]$.) 
Since $y=2^k$, we conclude $k\ge 10$ and $y\ge 1024$.
Choosing a fixed dimension $k$, the non-negativity constraints for the $A_i$ and $B_i$ form a polyhedron of dimension at most $2$, i.e., 
a polygon. As parameterizing variables, we use $A_8$ and $B_3$. Since $A_{16}\ge 0$, we have
$$ 
  B_3\ge 45 +\frac{10\cdot 4096}{y}.
$$
Since $A_{16}+4A_{40}\ge 0$, we have
$$
  B_3\le 87+\frac{2}{3} +\frac{2^{13}}{y}, 
$$
so that
$$
  A_8+A_{16}\le \frac{42+\frac{2}{3}}{4096}y-8.
$$
Since $A_{16}+A_{40}\ge 0$, we have
$$
  A_8\le \frac{y}{384}-2.
$$ 

\begin{Definition}
  Let $\mathcal{C}$ be a projective linear code of effective length $n$ in $\mathbb{F}_q^n$ with dimension $k=\dim(\mathcal{C})>2$. 
  For a pair of different codewords (of $\mathcal{C}$) let $\mathcal{T}:=\langle c_1,c_2\rangle=\{c_1,c_2,c_3\}$ denote the two-dimensional 
  subcode generated by $c_1$ and $c_2$. By $\left(\{w(c_1),w(c_2),w(c_3)\},\frac{w(c_1)+w(c_2)+w(c_3)}{2}\right)$ we denote the 
  \textit{subcode-type} of $\mathcal{T}$ Here, the first three numbers, read as a multiset, are the weights of the codewords of 
  $\mathcal{T}$ and the last number is the corresponding effective length.  
\end{Definition}

As an example, we consider a binary projective $2^3$-divisible code $\mathcal{C}$ of length $n=59$ and list the possible 
subcode-types containing a fixed codeword of length $40$: 
\begin{itemize}
  \item $e_1$: $(\{40,24,16\},40)$;
  \item $e_2$: $(\{40,32,8\},40)$;
  \item $e_3$: $(\{40,24,24\},44)$;
  \item $e_4$: $(\{40,32,16\},44)$;
  \item $e_5$: $(\{40,40,8\},44)$; 
  \item $e_6$: $(\{40,32,24\},48)$;
  \item $e_7$: $(\{40,40,16\},48)$; 
  \item $e_8$: $(\{40,40,24\},52)$; 
  \item $e_9$: $(\{40,32,32\},52)$; 
  \item $e_{10}$: $(\{40,40,32\},56)$.
\end{itemize}
Here $e_i$ denotes the number of times this subcode-type occurs in $\mathcal{C}$, where a codeword of weight $40$ is fixed. 
Counting the codewords of $\mathcal{C}$ by weight gives
\begin{eqnarray*}
  A_{40}[\mathcal{C}] &=& 1+e_5+e_7+e_8+e_{10}\\
  A_{32}[\mathcal{C}] &=& e_2+e_4+e_6+2e_9+e_{10}\\
  A_{24}[\mathcal{C}] &=& e_1+2e_3+e_6+e_8\\
  A_{16}[\mathcal{C}] &=&e_1+e_4+e_7\\
  A_{8}[\mathcal{C}]  &=& e_2+e_5,
\end{eqnarray*}
so that obviously $\sum_{i=1}^{10} e_i=2^{\dim(\mathcal{C})-1}-1=2^{k-1}-1$. In the remaining part of this section $\mathcal{C}$ always denotes 
a binary projective $2^3$-divisible code of length $n=59$.  

\begin{Lemma}
  \label{eq_e_19}
  Let $\mathcal{D}$ be the residual code of a codeword $c$ of weight $40$ of $\mathcal{C}$. Setting $z=2^{\dim(\mathcal{C})-\dim(\mathcal{D})-1}$, 
  we have
  \begin{eqnarray*}
    z-1 &=& e_1+e_2\\
    z\cdot A_4[\mathcal{D}] &=& e_3+e_4+e_5\\
    z\cdot A_8[\mathcal{D}] &=& e_6+e_7\\
    z\cdot A_{12}[\mathcal{D}] &=& e_8+e_9\\
    z\cdot A_{16}[\mathcal{D}] &=& e_{10}.
  \end{eqnarray*}
\end{Lemma}
\begin{Proof}
  Observe that the non-zero weights of $\mathcal{D}$ are contained in $\{4,8,12,16\}$, since $\mathcal{D}$ is $2^2$-divisible, and 
  count the $2$-dimensional subcodes of $\mathcal{C}$ containing codeword $c$ by their effective length.
\end{Proof}

\begin{Lemma}
  \label{lemma_exclude_19_conf_2}
  Let $\mathcal{C}$ be a projective $2^3$-divisible code of length $n=59$, $k=\dim(\mathcal{C})$ and $d$ be a codeword of weight $40$. 
  For the corresponding residual code $\mathcal{D}$ we have $w(\mathcal{D})\neq 1+78\cdot x^8+48\cdot x^{12}+1\cdot x^{16}$. 
\end{Lemma}
\begin{Proof}
  Assume the contrary. Since $A_{16}[\mathcal{D}]=1$ and $\dim(\mathcal{D})=7$, we have $e_{10}=2^{k-8}$, so that $A_{40}\ge 2^{k-8}+1$.
  Since $e_1+e_2=2^{k-8}-1$, we have $A_{8}+A_{16}\ge 2^{k-8}-1$. Thus, $A_{8}+A_{16}+A_{40}\ge 2^{k-7}$. From the above we compute 
  $A_{24}+A_{32}=5+2A_8+\frac{127}{128}y=5+2A_8+\frac{127}{128}\cdot 2^k$, so that 
  $A_{8}+A_{16}+A_{24}+A_{32}+A_{40}\ge 5+2A_8 +2^k> 2^k-1$, which is a contradiction.  
\end{Proof}

\begin{Lemma}
  \label{lemma_exclude_19_conf_3}
  Let $\mathcal{C}$ be a projective $2^3$-divisible code of length $n=59$, $k=\dim(\mathcal{C})$ and $d$ be a codeword of weight $40$. 
  For the corresponding residual code $\mathcal{D}_1$ we have $w(\mathcal{D}_1)\neq 1+4\cdot x^4+150\cdot x^8+100\cdot x^{12}+1\cdot x^{16}$. 
\end{Lemma}
\begin{Proof}
  Assume the contrary. Since $A_{16}[\mathcal{D}_1]=1$ and $\dim(\mathcal{D}_1)=8$, we have $e_{10}=2^{k-9}$, so that $A_{40}\ge 2^{k-9}+1$.
  Since $e_1+e_2=2^{k-9}-1$, we have $A_{8}+A_{16}\ge 2^{k-9}-1$. Now we switch to the geometric language. By $\overline{\mathcal{C}}$ we 
  denote the set of $59$ points in $\mathbb{F}_2^k$ corresponding to $\mathcal{C}$ and by $\overline{\mathcal{D}}_1$ we 
  denote the set of $19$ points corresponding to $\mathcal{D}$, i.e., there exists a hyperplane $H_1$ of $\mathbb{F}_2^k$ with 
  $\overline{\mathcal{D}}_1=H_1\cap \overline{\mathcal{C}}$. Since $e_{10}\ge 1$, there exists another hyperplane $H_2$ such that 
  $\#\overline{\mathcal{D}}_2=19$ and $\#\overline{\mathcal{D}}_1\cap\overline{\mathcal{D}}_2=3$, where 
  $\overline{\mathcal{D}}_2=H_2\cap\overline{\mathcal{C}}$. $e_{10}=2^{k-9}$ counts the number of subspaces $S$ of codimension $2$ of 
  $\mathbb{F}_2^k$ with $S\le H_1$ and $\#S\cap \overline{\mathcal{C}}=3$. One of these, call it $S'$, is also contained in $H_2$, i.e., 
  $S'=H_1\cap H_2$. Let $\mathcal{D}_2$ denote the residual code corresponding to $\overline{\mathcal{D}}_2$. Since $S'\cap\overline{\mathcal{C}}$ 
  corresponds to a codeword of weight $16$ in $\mathcal{D}_2$. Thus, $H_2$ contains $2^{k-1-\dim(\mathcal{D}_2)}\ge 2^{k-9}$ subspaces $S$ of 
  codimension $2$ of $\mathbb{F}_2^k$ with $S\le H_2$ and $\#S\cap \overline{\mathcal{C}}=3$. 
  In total we have at least $2^{k-8}-1$ different subspaces $S_i$ of codimension $2$ of $\mathbb{F}_2^k$ with $\#S\cap \overline{\mathcal{C}}=3$ and either 
  $S\le H_1$ or $S\le H_2$. ($H_1\cap H_2$ determines a unique such subspace.) Let us number these subspaces in such a way that $S_1=H_1\cap H_2$, 
  $S_i\le H_1$, $S_i\not\le H_2$ for $2\le i\le 2^{k-9}$ and $S_i\le H_2$, $S_i\not\le H_1$ for $2^{k-9}+1 \le i \le 2^{k-8}-1$. Each of these 
  subspaces $S_i$ determines two unique hyperplanes $H_{i,1}$ and $H_{i,2}$ of $\mathbb{F}_2^k$ with $S_i\le H_{i,j}$ and 
  $\#H_{i,j}\cap\overline{\mathcal{C}}=19$. W.l.o.g.\  we assume $H_{1,1}=H_1$, $H_{1,2}=H_2$, $H_{1,1}=H_1$ for $2\le i\le 2^{k-9}$ and 
  $H_{i,1}=H_2$ for $2^{k-9}+1\le i \le 2^{k-8}-1$. With this, the hyperplanes $H_1$, $H_2$, and $H_{i,2}$ for $2\le i\le 2^{k-8}-1$ 
  are pairwise disjoint by construction and contain exactly $19$ points of $\overline{\mathcal{C}}$ each. Thus, we have 
  $A_{40}[\mathcal{C}]\ge 2^{k-8}$. 
  
  Let $H_1$ and $H_2$ be the same two hyperplanes as above. Since $e_1+e_2=2^{k-9}-1$, there are $2^{k-9}-1$ subspaces $T$ of codimension 
  $2$ of $\mathbb{F}_2^k$ with $T\le H_1$ and $\# T\cap\overline{\mathcal{C}}=19$. Since $\dim(\mathcal{D}_2)\le 8$ there are 
  (at least) $2^{k-9}-1$ subspaces $T$ of codimension $2$ of $\mathbb{F}_2^k$ with $T\le H_2$ and $\# T\cap\overline{\mathcal{C}}=19$. 
  Since $\#\overline{\mathcal{D}}_1\cap\overline{\mathcal{D}}_2=3$, we have at least $2^{k-8}-2$ such subspaces $T_1,\dots, T_{2^{k-8}-2}$. 
  Each of these subspaces is either contained in a hyperplane containing $43$ or $51$ points of $\overline{\mathcal{C}}$. For 
  $1\le i,j\le 2^{k-8}-2$ and $i\neq j$ assume that there is a hyperplane $H$ of $\mathbb{F}_2^k$ with $T_i\le H$, $T_j\le H$, and 
  $\#H\cap\overline{\mathcal{C}}\ge 43$. W.l.o.g.\ we assume $T_i\le H_1$ and $T_j\le H_2$. (Observe that $\dim(T_i\cup T_j)=k-1$, so that 
  e.g.\ $S_i\le H_1$ and $S_j\le H_1$ implies $\langle S_i,S_j\rangle=H_1=H$.)   
  Now, we argue that $\#H\cap\overline{\mathcal{C}}=51$. Assume, to the contrary, that $\#H\cap\overline{\mathcal{C}}=43$. 
  Since $\#\overline{\mathcal{D}}_1=19$, $\#\overline{\mathcal{D}}_2=19$, and $\#\overline{\mathcal{D}}_1\cap \overline{\mathcal{D}}_2=3$, 
  we have $\#\overline{\mathcal{C}}'=32$, where $\overline{\mathcal{C}}'= \overline{\mathcal{D}}_1\cup \overline{\mathcal{D}}_2 - 
  \overline{\mathcal{D}}_1\cap \overline{\mathcal{D}}_2$. Note that $\overline{\mathcal{C}}'$ is a $2^2$-divisible set, so that 
  $H\cap\overline{\mathcal{C}}-\overline{\mathcal{C}}'$ is also a $2^2$-divisible set. However, there is no $2^2$-divisible set of 
  cardinality $43-32=11$, which is a contradiction. Thus, $A_{16}\ge 2\cdot\left(2^{k-9}-1\right)-2A_8$, which can be rewritten to 
  $2A_8+A_{16}\ge 2^{k-8}-2$. Thus, we have 
  $$
    -6-A_8+2^{k-7}=2A_8+A_{16}+A_{40}\ge 2^{k-8}-2+2^{k-8}=2^{k-7}-2,
  $$    
  so that $A_8\le -4$, which is a contradiction.   
\end{Proof}


\begin{Lemma}
  \label{lemma_exclude_19_k_10}
  There is no projective $2^3$-divisible code of length $n=59$, dimension $k\le 10$ that contains a codeword of weight $40$.
\end{Lemma}
\begin{Proof}
  As shown at the beginning of this section, we have $k\ge 10$, so that we can assume $k=10$. Now, assume the contrary and let $\mathcal{D}$ be the 
  residual code of a codeword of weight $40$. Due to Lemma~\ref{lemma_we_2_2_n_19} 
  and Lemma~\ref{lemma_exclude_19_conf_3} we have $\dim(\mathcal{D})=7$. Since $e_1+e_2=2^{k-8}-1$, we have $A_{8}+A_{16}\ge 2^{k-8}-1$. However, this 
  contradicts $A_8+A_{16}\le \frac{42+\frac{2}{3}}{4096}y-8=\frac{8}{3}<3$, which was concluded above from the first four MacWilliams identities.
\end{Proof}

\begin{Lemma}
  \label{lemma_secants}
  If there is projective $2^3$-divisible code of length $n=59$, then there also exists such a code with dimension $k=10$.
\end{Lemma}
\begin{Proof}
  Let $\mathcal{C}$ be a projective $2^3$-divisible code of length $n=59$ with minimum possible dimension $k$ and $\mathcal{P}$ be the corresponding 
  set of points in $\mathbb{F}_2^k$, i.e., the points spanned by the columns of an arbitrary generator matrix of $\mathcal{C}$. For each point $Q$ in 
  $\mathbb{F}_2^k$ that is not contained in $\mathcal{C}$ we consider the projection through $Q$, that is the multiset image of $\mathcal{P}$ under 
  the map $\mathbb{F}_2^k \to \mathbb{F}_2^k/Q$, $v \mapsto (v+Q) / Q$. The result is a multiset of points in $\mathbb{F}_2^{k-1}$ where all points 
  on a common line through $Q$ got identified. The corresponding linear code $\mathcal{C}'$ is a subcode of $\mathcal{C}$ and therefore a $2^3$-divisible 
  binary linear code of (effective) length $n=59$, since $Q$ is not contained in $\mathcal{P}$, and dimension $k-1$. The code $\mathcal{C}'$ is non-projective 
  iff every point $Q$ not contained in $\mathcal{P}$ lies on a secant, i.e., on a line consisting of $Q$ and two points in $\mathcal{P}$. Due to the assumed 
  minimality of $k$, any point of $\mathbb{F}_2^k\backslash\mathcal{P}$ must lie on a secant. $\mathcal{P}$ admits at most 
  $\binom{\#\mathcal{P}}{2} = {{59}\choose {2}}=1711$ secants which cover at most $1711$ different points not in $\mathcal{P}$.
  On the other hand, there are $2^k - 60$ points not contained in $\mathcal{P}$, which forces $k\le 10$. 
  As shown at the beginning of this section, we have $k\ge 10$, so that $k=10$.  
\end{Proof}

\begin{Theorem}
  \label{main_theorem}
  There is no binary projective $2^3$-divisible linear code of effective length $59$.
\end{Theorem}
\begin{Proof}
  Assume that $\mathcal{C}$ is a binary projective $2^3$-divisible linear code of effective length $n=59$. Due to Lemma~\ref{lemma_secants} 
  we can assume that $\mathcal{C}$ has dimension $k=10$. As mentioned in Section~\ref{sec_preliminaries}, the weights are contained 
  in $\{8,16,24,32,40\}$. Weight $40$ is excluded in Lemma~\ref{lemma_exclude_19_k_10}. Plugging $A_{40}[\mathcal{C}]=0$ and $y=2^{10}$ 
  into the equations at the beginning of this section gives $A_{16}[\mathcal{C}]=2-3A_{8}[\mathcal{C}]$, where $A_{8}[\mathcal{C}]$ is considered
  as the free parameter in the solution of the first four MacWilliams identities. Since $A_{8}[\mathcal{C}]$ and $A_{16}[\mathcal{C}]$ are non-negative integers, we conclude that $A_{8}[\mathcal{C}]=0$ and obtain the unique solution 
  $A_{16}[\mathcal{C}]=2$, $A_{24}[\mathcal{C}]=312$, $A_{32}[\mathcal{C}]=709$.

  Now let $c$ be one of the two codewords of weight $16$ and $\mathcal{C}'$ be the restriction of $\mathcal{C}$ to $\operatorname{supp}(c)$, i.e., to the $16$ non-zero positions of $c$.
  The code $\mathcal{C}'$ is a binary linear code of effective length $16$.
  By the $2^3$-divisibility of $\mathcal{C}$ and the fact that $\mathcal{C}'$ contains the all-$1$-word, we see that $\mathcal{C}'$ is $2^2$-divisible.

  Consider a codeword $x\in\mathcal{C}$ whose restriction to $\operatorname{supp}(c)$ is the zero word.
  Then $w(x + c) = w(x) + 16 \leq 32$ and therefore $w(x) \leq 16$.
  Thus, $x$ is either the zero word of $\mathcal{C}$ or, possibly, the second word of weight $16$.
  This implies $\dim(\mathcal{C}') \geq k-1 = 9$.%
  \footnote{To make this argument precise, apply the rank-nullity theorem to the surjective linear map $\varphi : \mathcal{C} \to \mathcal{C}'$ which restricts a codeword to the $16$ coordinates in $\operatorname{supp}(c)$.}
  However, the $2^2$-divisibility implies that $\mathcal{C}'$ is self-orthogonal of length $16$ and therefore of dimension at most $16/2 = 8$ -- contradiction.
\end{Proof}


\section{Conclusion and open problems}
\label{sec_conclusion}
By purely theoretical methods we were able to exclude the existence of a binary projective $2^3$-divisible linear code 
of effective length $n=59$. This completes the characterization of the possible lengths of $2^3$-divisible projective linear codes, which play 
some role in applications. 

We might have streamlined our theoretical reasoning. Lemma~\ref{lemma_exclude_19_conf_2} may be removed completely and the proof of 
Lemma~\ref{lemma_exclude_19_conf_3} may be restricted to the case $k=10$. However, we have included their full proofs due to the subsequent
reason. In Lemma~\ref{lemma_exclude_19_conf_2} we have used a codeword of weight $40$ to conclude the existence of a certain number of codewords 
of weight at most $16$ that contradicts the MacWilliams identities. In Lemma~\ref{lemma_exclude_19_conf_3} we have used a refined analysis starting 
from two codewords of weight $40$. In principle each codeword of weight $40$ implies the existence of some codewords of weight at most $16$. Since 
there have to be a lot of codewords of weight $40$, there should also be quite some codewords of weight at most $16$. It would be very nice to turn 
this vague idea into a rigorous theoretic argument for the exclusion of the third possible weight enumerator in Lemma~\ref{lemma_we_2_2_n_19}.  

While we prefer a theoretical argument over computer computations, we nevertheless state an alternative computational verification of our 
main result in the appendix. The computational methods might also be applicable for other coding theoretic existence questions.

Clearly, it would be desirable to see generalizations of the (completed) characterization Theorem~\ref{thm_characterization} for other parameters. To this end, we 
state the list of length of $2^4$-divisible binary projective linear codes for which the existence question is undecided, at least up to our knowledge:
$$
  \{130,163,164,165,185,215,216,232,233,244,245,246,247,274,275,277,278,306,309\}.
$$   
For $q=3$ the smallest open case is that of $3^2$-divisible ternary projective linear codes with 
$$\{70,77,99,100,101,102,113,114,115,128\}$$ 
as the set of undecided lengths.

\section*{Acknowledgments}
We would like to thank Iliya Bouyukliev for discussions on the usage of his software \texttt{Q-Extension} and a personalized version that is capable to deal 
with larger dimensions, see the appendix. 


\appendix
\section{An alternative computational approach}
Instead of the theoretical reasoning presented in Section~\ref{sec_no_59}, Theorem~\ref{main_theorem} can also be 
obtained by computer calculations. Our starting point is a computational version of Lemma~\ref{lemma_we_2_2_n_19}. In \cite{doran2011codes} 
binary projective $2^2$-divisible linear codes with small parameters were classified using extensive computer enumerations. For length $19$ 
there are exactly three non-isomorphic examples, see \url{https://rlmill.github.io/de_codes}, which of course match our three geometric 
constructions given in Section~\ref{sec_preliminaries}. We have verified the specific case $n=19$ using \texttt{Q-Extension} \cite{bouyukliev2007q}.


\begin{table}[htp]
  \begin{center}
    \begin{tabular}{rrrrrrrrrrrrrrrrrrrr}
      \hline
      k/n & 8 & 12 & 14 & 15 & 16 & 20 & 22 & 23 & 24 & 26 & 27 & 28 & 29 & 30 & 31 & 32 & 34 & 35 & 36 \\ 
      \hline
      1   & 1 &  0 &  0 &  0 &  1 &  0 &  0 &  0 &  1 &  0 &  0 &  0 &  0 &  0 &  0 &  1 &  0 &  0 &  0 \\ 
      2   &   &  1 &  0 &  0 &  1 &  1 &  0 &  0 &  2 &  0 &  0 &  2 &  0 &  0 &  0 &  3 &  0 &  0 &  3 \\
      3   &   &    &  1 &  0 &  1 &  1 &  1 &  0 &  3 &  1 &  0 &  4 &  0 &  3 &  0 &  8 &  2 &  0 &  9 \\ 
      4   &   &    &    &  1 &  1 &  1 &  1 &  1 &  4 &  1 &  1 &  6 &  1 &  6 &  4 & 18 &  7 &  3 & 27 \\
      5   &   &    &    &    &  1 &  0 &  0 &  1 &  4 &  2 &  1 &  7 &  1 &  8 &  8 & 32 & 14 &  7 & 54 \\
      6   &   &    &    &    &    &    &    &    &  1 &  0 &  1 &  6 &  2 &  7 &  8 & 34 & 11 &  7 & 65 \\
      7   &   &    &    &    &    &    &    &    &    &    &    &  1 &  1 &  6 &  6 & 24 &  5 &  3 & 36 \\
      8   &   &    &    &    &    &    &    &    &    &    &    &    &    &  2 &  4 & 13 &  1 &  1 & 11 \\
      9   &   &    &    &    &    &    &    &    &    &    &    &    &    &    &  1 &  5 &  0 &  0 &  1 \\
     10   &   &    &    &    &    &    &    &    &    &    &    &    &    &    &    &  1 &  0 &  0 &  0 \\
      \hline 
    \end{tabular}
    \caption{Number of $[n,k]_2$ codes with weights in $\{8,16,24,32,40\}$ -- part 1.}
    \label{tab_8_16_24_32_40_p1}
  \end{center}
\end{table}

The next step is to show that none of the three mentioned codes above can be a residual code of a binary projective 
$2^3$-divisible code of length $n=59$. Next we provide all computational results for the case of the weight enumerator 
$w(\mathcal{C})=1+1\cdot x^4+75\cdot x^8+51\cdot x^{12}$. First note that the unique isomorphism type 
can be represented by the following generator matrix:
$$
  M_{19}=
  \begin{pmatrix}
    1 0 0 0 0 0 1 1 1 1 0 0 0 1 1 0 1 0 0 \\
    0 1 0 0 0 0 1 1 1 0 0 1 0 1 0 1 0 1 0 \\
    0 0 1 0 0 0 1 1 0 0 1 1 0 0 1 1 1 0 0 \\
    0 0 0 1 0 0 0 1 0 1 1 0 1 0 1 0 1 1 0 \\
    0 0 0 0 1 0 1 0 0 1 1 0 0 1 1 1 0 1 0 \\
    0 0 0 0 0 1 1 0 0 1 1 0 0 1 0 1 1 1 0 \\
    0 0 0 0 0 0 1 0 1 0 0 0 1 0 1 1 1 1 1 \\
  \end{pmatrix}
$$ 

With this, the generator matrix of a $2^3$-divisible projective linear code $C_{59}$ of length $n=59$, that contains a codeword of weight $40$, can be written as
$$
\begin{pmatrix}
  M_{19} & M_1\\ 
       0 & M_2
\end{pmatrix},
$$
where $M_2$ is the generator matrix of a $2^3$-divisible linear code $C_{40}$ of length $n=40$ that is not necessarily projective. Using the 
software package \texttt{Q-Extension} \cite{bouyukliev2007q} we have enumerated all binary linear codes with effective length $n\le 40$ and 
weights in $\{8,16,24,32,40\}$, see Table~\ref{tab_8_16_24_32_40_p1} and Table~\ref{tab_8_16_24_32_40_p2}. Here blank entries on the left hand side 
of each row as well as missing columns correspond to effective lengths where no such code exists. We remark that the $2^3$-divisible linear codes 
up to length $n=48$ have been determined, in principle, in \cite{betsumiya2012triply}. We have validated our results with the corresponding list of 
codes at http://www.st.hirosaki-u.ac.jp/$\sim$betsumi/triply-even/.  

\begin{table}[htp]
  \begin{center}
    \begin{tabular}{rrrrrrrrrrrrrrrrrrrr}
      \hline
      k/n &  37 & 38 & 39 &  40 \\ 
      \hline
      1   &   0 &  0 &  0 &   1 \\ 
      2   &   0 &  0 &  0 &   4 \\
      3   &   0 &  6 &  0 &  17 \\ 
      4   &   2 & 22 & 10 &  64 \\
      5   &   5 & 59 & 36 & 194 \\
      6   &   8 & 79 & 57 & 347 \\
      7   &   5 & 61 & 49 & 323 \\
      8   &   1 & 21 & 30 & 177 \\
      9   &   0 &  2 & 10 &  59 \\
     10   &   0 &  0 &  1 &  11 \\ 
     11   &   0 &  0 &  0 &   1 \\
      \hline 
    \end{tabular}
    \caption{Number of $[n,k]_2$ codes with weights in $\{8,16,24,32,40\}$ -- part 2.}
    \label{tab_8_16_24_32_40_p2}
  \end{center}
\end{table}

Thus, the dimension of $C_{40}$ is at most $11$, so that the dimension of $C_{59}$ is at most $18$. We remark that the \textit{divisible code bound} 
from \cite{ward1992bound} gives an upper bound of $21$ for the dimension. In order to enumerate the possibilities for $C_{59}$ we can start 
from one of the codes $C_{40}$ and iteratively add one row to the generator matrix maintaining the property that the code is $2^3$-divisible. We start by formulating 
this general approach as an enumeration problem of integral points in a polyhedron:
\begin{Lemma}
  \label{lemma_ILP} 
  Let $G$ be a systematic generator matrix of an $[n,k]_2$ code whose weights are $\Delta$-divisible and are contained in $\left[a\cdot \Delta,b\cdot \Delta\right]$. 
  By $c(u)$ we denote the number of columns of $G$ that equal $u$ for all $u$ in $\mathbb{F}_2^k\backslash\mathbf{0}$, $c(\mathbf{0})=n'-n$,  and let $\mathcal{S}(G)$ be 
  the set of feasible solutions of 
  \begin{eqnarray}
    \Delta y_h+\sum_{v\in\mathbb{F}_2^{k+1}\,:\,v^\top h=0} x_v =n-a\Delta&&\forall h\in\mathbb{F}_2^{k+1}\backslash\mathbf{0}\label{eq_hyperplane}\\
    x_{(u,0)}+x_{(u,1)} =c(u) && \forall u\in\mathbb{F}_2^{k} \label{eq_c_sum}\\
    x_{e_i}\ge 1&&\forall 1\le i\le k+1\label{eq_systematic}\\
    x_v\in \mathbb N &&\forall v\in\mathbb{F}_2^{k+1}\\ 
    y_h\in\{0,...,b-a\} && \forall h\in\mathbb{F}_2^{k+1}\backslash\mathbf{0}\label{hyperplane_var},
  \end{eqnarray}
  where $e_i$ denotes the $i$th unit vector in $\mathbb{F}_2^{k+1}$ and $n'\ge n+1$. Then, for every systematic generator matrix $G'$ of an $[n',k+1]_2$ code $C'$ 
  whose first $k$ rows coincide with $G$ we have a solution $(x,y)\in\mathcal{S}(G)$ such that $G'$ has exactly $x_v$ columns equal to $v$ for each $v\in\mathbb{F}_2^{k+1}$.
\end{Lemma}
\begin{proof}
  Let such a systematic generator matrix $G'$ be given and $x_v$ denote the number of columns of $G'$ that equal $v$ for all $v\in\mathbb{F}_2^{k+1}$. Since $G'$ is systematic, 
  Equation~(\ref{eq_systematic}) is satisfied. As $G'$ arises by appending a row to $G$, also Equation~(\ref{eq_c_sum}) is satisfied. Obviously, the $x_v$ are non-negative 
  integers. The conditions (\ref{eq_hyperplane}) and (\ref{hyperplane_var}) correspond to the restriction that the weights are $\Delta$-divisible and contained in 
  $\{a\Delta,\dots,b\Delta\}$.   
\end{proof}
We remark that also every solution in $\mathcal{S}(G)$ corresponds to an $[n',k+1]_2$ code $C'$ with generator matrix $G'$ containing $C$ as a subcode. Different generator 
matrices of corresponding isomorphic codes may be reduced to just one representing generator matrix using the tools from \texttt{Q-Extension} \cite{bouyukliev2007q}. In order 
to implement our extra knowledge of the generator matrix $M_{19}$ of the residual code, we ensure that the extended generator matrices coincide outside of the $40$ columns of 
$C_{40}$ with the last $r$ rows of $M_{19}$, where $r$ increases from $1$ to $7$. In terms of the integer linear program from Lemma~\ref{lemma_ILP} we just 
fix the counts $x_v$ of the corresponding columns $v$. At $r=7$ we can check if the resulting codes 
are projective. It turns out that this never happens, so that:

\begin{Lemma}
  \label{lemma_exclude_19_conf_1}
  Let $\mathcal{C}$ be a projective $2^3$-divisible code of length $n=59$, $k=\dim(\mathcal{C})$ and $d$ be a codeword of weight $40$. 
  For the corresponding residual code $\mathcal{D}$ we have $w(\mathcal{D})\neq 1+1\cdot x^4+75\cdot x^8+51\cdot x^{12}$. 
\end{Lemma}

Without giving the computational details, we remark that we have performed the described computational approach also for the two other residual codes, 
so that we obtain:
\begin{Corollary}
  \label{cor_no_weight_40}
  Let $\mathcal{C}$ be a projective $2^3$-divisible code of length $n=59$. Then all non-zero weights of $\mathcal{C}$ are contained 
  in $\{8,16,24,32\}$.
\end{Corollary}

We remark that we might have used inequalities coming from the MacWilliams identities, like $A_8+A_{16}\le \frac{42+\frac{2}{3}}{4096}y-8$, to 
cut the generation tree of linear codes. I.e.\ at every point where we have constructed a new subcode, we may count the number of codewords $A_8$ and $A_{16}$ 
of weight $8$ and $16$, respectively. The number of codewords of weight at most $16$ of the final code of length $n=59$ cannot decrease. The starting code 
$C_{40}$ also fixes the dimension of the final code to $k=\dim(C_{40})+7$, so that we may check the violation of the above inequality using $y=2^k$. However, 
the computer enumerations were fast enough so that it was not necessary to implement these checks.


\begin{table}[htp]
  \begin{center}
    \begin{tabular}{rrrrrrrrrrrrrrrrrrrr}
      \hline
      k/n &  40 & 41 &  42 &  43 &   44 &  45 &   46 &   47 &   48 &   49 &   50 \\ 
      \hline 
      2   &   2 &  0 &   0 &   0 &    1 &   0 &    0 &    0 &    1 &    0 &    0 \\
      3   &  11 &  0 &   6 &   0 &    9 &   0 &    5 &    0 &    6 &    0 &    3 \\ 
      4   &  49 &  2 &  36 &  12 &   62 &   9 &   50 &   15 &   59 &    9 &   42 \\
      5   & 154 & 12 & 158 &  67 &  316 &  62 &  362 &  146 &  503 &  135 &  495 \\
      6   & 277 & 29 & 335 & 180 &  903 & 273 & 1394 &  704 & 2533 &  877 & 3245 \\
      7   & 241 & 26 & 356 & 215 & 1194 & 528 & 2663 & 1829 & 6038 & 2176 & 8341 \\
      8   & 119 & 12 & 176 & 147 &  738 & 513 & 2285 & 2534 & 6753 & 2443 & 6643 \\
      9   &  29 &  3 &  49 &  57 &  244 & 257 & 1000 & 1610 & 3535 & 1416 & 1876 \\
     10   &   4 &  0 &   3 &  12 &   41 &  54 &  229 &  499 &  144 &  480 &   50 \\
     11   &     &    &     &   1 &    3 &  14 &   51 &   90 &   98 &   69 &   36 \\ 
     12   &     &    &     &     &      &   2 &    8 &   14 &    7 &   15 &    4 \\  
     13   &     &    &     &     &      &     &    1 &    1 &    0 &    1 &    0 \\ 
      \hline 
    \end{tabular}
    \caption{Number of $[n,k]_2$ codes with weights in $\{8,16,24,32\}$ -- part 1.}
    \label{tab_8_16_24_32_p1}
  \end{center}
\end{table}

At this point we know if $\mathcal{C}_{59}$ is a projective $2^3$-divisible code of length $n=59$, then its weights are upper bounded by $32$. So we enumerate those 
linear codes, that are possibly non-projective, using \texttt{Q-Extension} \cite{bouyukliev2007q}. Of course, for $n<40$ the counts coincide with those
from Table~\ref{tab_8_16_24_32_40_p1} and Table~\ref{tab_8_16_24_32_40_p2}. For length $n\ge 40$ we state the resulting counts in Table~\ref{tab_8_16_24_32_p1} 
and Table~\ref{tab_8_16_24_32_p2}.

\begin{table}[htp]
  \begin{center}
    \begin{tabular}{rrrrrrrrrrrrrrrrrrrr}
      \hline
      k/n &   51 &    52 &   53 &   54 &    55 &    56 &    57 &    58 &   59 \\ 
      \hline 
      2   &    0 &     0 &    0 &    0 &     0 &     0 &     0 &     0 &    0 \\
      3   &    0 &     2 &    0 &    1 &     0 &     1 &     0 &     0 &    0 \\ 
      4   &   16 &    30 &   10 &   15 &     7 &     7 &     4 &     2 &    1 \\
      5   &  210 &   444 &  206 &  291 &   157 &   144 &    74 &    40 &   17 \\
      6   & 1713 &  3695 & 2252 & 3160 &  2523 &  2295 &  1585 &   886 &  334 \\
      7   & 4585 & 10523 & 7026 & 9634 & 10043 & 11322 & 10393 &  8195 & 3695 \\
      8   & 3877 &  7309 & 5860 & 6852 &  9477 & 12719 & 14811 & 15536 & 9237 \\
      9   & 1238 &  1643 & 1975 & 2132 &  3655 &  6134 &  7659 &  8871 & 6965 \\
     10   &  198 &   171 &  305 &  224 &   768 &   547 &   558 &   911 & 2905 \\
     11   &   10 &    16 &   48 &   65 &    69 &   259 &   204 &   295 &  675 \\ 
     12   &    1 &     1 &    6 &   10 &     7 &    31 &    48 &    56 &  174 \\  
     13   &    0 &     0 &    1 &    1 &     0 &     3 &     7 &     8 &   34 \\
     14   &      &       &      &      &       &       &     1 &     1 &    6 \\
     15   &      &       &      &      &       &       &       &       &    1 \\ 
      \hline 
    \end{tabular}
    \caption{Number of $[n,k]_2$ codes with weights in $\{8,16,24,32\}$ -- part 2.}
    \label{tab_8_16_24_32_p2}
  \end{center}
\end{table}

Having checked that all resulting codes of length $n=59$ are non-projective we conclude our main result Theorem~\ref{main_theorem}.

We remark that the \textit{divisible code bound} from \cite{ward1992bound} gives an upper bound of $17$ for binary linear codes 
with weights in $\{8,16,24,32\}$.

\end{document}